\documentclass[12pt,reqno]{amsart}
\usepackage[all]{xy}   
\usepackage {amssymb,latexsym,amsthm,amsmath,mathtools}
\usepackage{xcolor}
\usepackage{amsthm}
\usepackage{mathtools}
\usepackage{tikz-cd}
\usepackage[normalem]{ulem}
\usepackage[raiselinks=false,colorlinks=true,citecolor=blue,urlcolor=blue,
linkcolor=blue,bookmarksopen=true,pdftex]{hyperref}

\newcommand{\irr}{\textup{Irr}}

\makeatletter
\def\imod#1{\allowbreak\mkern10mu({\operator@font mod}\,\,#1)}
\makeatother

\newtheorem{theorem}{Theorem}[section]
\newtheorem{lemma}[theorem]{Lemma}

\newtheorem{conjecture}{Conjecture}
\newtheorem*{theorem*}{Theorem A}
\theoremstyle{definition}

\numberwithin{equation}{section}

\newcommand{\ignore}[1]{}

\newcommand{\mynote}[1]{}
\newcommand{\genlegendre}[4]{%
	\genfrac{(}{)}{}{#1}{#3}{#4}%
	\if\relax\detokenize{#2}\relax\else_{\!#2}\fi
}

\usepackage{ytableau}

\usepackage{tikz}
\usepackage{float}
\begin{document}
	\setcounter{section}{0}
	
	\title[Vanishing pairs of conjugacy classes]{Vanishing pairs of conjugacy classes\\ for the symmetric group}
	\author{Velmurugan S}
	\address{The Institute of Mathematical Sciences, Chennai,}
     \address{Homi Bhabha National Institute, Mumbai,}
     \email{velmurugan.math@gmail.com}

\begin{abstract}
    In this short note, we classify pairs of conjugacy classes of the symmetric group such that any non-linear irreducible character of the symmetric group vanishes on at least one of them.
\end{abstract}
\maketitle

\emph{Keywords:} symmetric group, character.

\emph{AMS Subject Classification}: Primary 20C30; Secondary 20C15.

\section{Introduction}
Let $G$ be a finite group.
Let $\irr(G)$ denote the set of all irreducible characters of $G$.
We say that an element $g\in G$ is a vanishing element if there exists an irreducible character $\chi$ of $G$ such that $\chi(g)=0$ and, in this case, we call the conjugacy class containing $g$ a vanishing conjugacy class of $G$ and call $g$ as 
a zero of $\chi$.
The study of finding zeros of irreducible characters of $G$ goes back to Burnside who proved~\cite[Theorem 3.15]{Isaacs_I_Martin} that any non-linear irreducible character (i.e., $\chi(1)>1$) of $G$ vanishes on some element of $G$.

It is natural to ask that whether there is an element $g$ in $G$ such that every non-linear irreducible character vanishes on $g$.
It turns out that if so, then the group $G$ must be solvable, see~\cite{Frieder_Ladisch}.
Generally, for the description of a group which has one element where every non-linear irreducible characters vanishes, see~\cite{Mark_Lewis_L}.

For a finite group $G$, let $k(G)$ denote the least positive integer $k$ such that there exist $k$-many conjugacy classes where any non-linear irreducible character vanishes on at least one of them.
Lewis et al.~\cite{Mark_Lucia_Emanuele_Lucia_Hung} conjectured the following:
\begin{conjecture}\label{conjecture:Mark_Lucia_Emanuele_Lucia_Hung}
    Let $G$ be a finite group.
    Then one of the following holds:
    \begin{enumerate}
        \item $k(G)\leq 3$.
        \item If $G$ is solvable, then $k(G)\leq 2$.
    \end{enumerate}
\end{conjecture}
They proved the conjecture for finite non-abelian simple groups.
Furthermore, they showed that, for the symmetric group $S_n$, $k(S_n)=2$.
In this note, we classify such pairs of conjugacy classes of $S_n$.
For every partition $\mu$ of $n$, let $C_\mu$ denote the conjugacy class of $S_n$ consisting of permutations with cycle type $\mu$.
\begin{theorem*}
    Let $n$ be a positive integer greater than $6$.
    Let $\mu$ and $\nu$ be partitions of $n$.
    Then any non-linear irreducible character vanishes on at least one of $C_\mu$ or $C_\nu$ if and only if $\{\mu,\nu\}=\{(n),(n-1,1)\}$.
    Equivalently, for any two conjugacy classes $C_\mu$ and $C_\nu$ of $S_n$, there exists a non-linear irreducible character $\chi$ of $S_n$ such that $\chi(w_\mu)\neq 0$ and $\chi(w_\nu)\neq 0$ if and only if $\{\mu,\nu\}\neq \{(n),(n-1,1)\}$.
\end{theorem*}
The above theorem turns out to be a counterpart of the following theorem of Hung, Moret\'o and Morotti.
\begin{theorem}{\cite[Theorem A]{Hung}}
    Let $n>8$ be a positive integer.
    Then for non-linear characters $\chi_\lambda$ and $\chi_\nu$ of $S_n$, they have a common zero if and only if  $\{\chi_\lambda(1),\chi_\mu(1)\}\neq\{n(n-3)/2, (n-1)(n-2)/2\}$.
\end{theorem}
For more results and conjectures on the zeros of irreducible characters of finite groups, we refer the reader to the survey article~\cite{MR0002127}.
\section{Proof of the theorem}
Let $n$ be a positive integer greater than $10$ (for $n\leq 10$, we can verify the theorem by direct computation).
Let $S_n$ be the symmetric group on $n$ letters.
We recall that the irreducible characters of $S_n$ are indexed by partitions of $n$.
For a partition $\lambda=(\lambda_1,\lambda_2,\dotsc,\lambda_r)$ of $n$, let $\chi_\lambda$ denote the irreducible character of $S_n$ indexed by $\lambda$.
Let $w_\lambda$ be a representative of the conjugacy class $C_\lambda$.

Firstly, we shall prove the 'only if' part of the theorem.
The proof uses Murnaghan-Nakayama rule~\cite[2.4.7]{James_Kerber} which is as follows.
\begin{lemma}[Murnaghan-Nakayama rule]
    \label{lemma:Murnaghan_Nakayama}
Let $\lambda,\mu$ be partitions of $n$.
Then
\begin{displaymath}
    \label{eq:Murnaghan_Nakayama}
    \chi_\lambda(w_\mu)=\sum\limits_{b\in \lambda : h(b)=\mu_t} (-1)^{\text{ht}(\text{rim}_b)-1} \chi_{\lambda-\text{rim}_b}(w_{\bar{\mu}_t}),
\end{displaymath}
where the summation varies over all the cells $b$ of the Young diagram of $\lambda$ with hook length $h_b$ equal to a fixed part $\mu_t$ of $\mu$, $\text{rim}_b$ is the rim hook of $b$, $\text{ht}(\text{rim}_b)$ is the number of rows $\text{rim}_b$ contains and $\bar{\mu}_t$ is the partition obtained from $\mu$ by removing the part $\mu_t$.
\end{lemma}

Now using Murnaghan-Nakayama rule, we see that $\chi_\lambda$ vanishes on the conjugacy class $C_{(n)}$ except when $\lambda$ is a hook partition (i.e., $\lambda_2\leq 1$).
Again using Murnaghan-Nakayama rule, we see that all irreducible characters $\chi_\lambda$ corresponding to the hook partitions of $n$ vanishes on the conjugacy class $C_{(n-1,1)}$ except for the linear characters $\chi_{(n)}$ and $\chi_{(1^n)}$.
This completes the proof of the 'only if' part of the theorem.  

Now we prove the 'if' part of the theorem.
Let $C_\mu$ and $C_\nu$ be two conjugacy classes (not necessarily distinct) of $S_n$ such that any non-linear irreducible character of $S_n$ vanishes on at least one of them.
By Schur's second orthogonality relations, we have
\begin{equation}\label{eq:Schur_scond}
   1+ \chi_{(1^n)}(w_\mu) \chi_{(1^n)}(w_\nu)= \sum\limits_{\lambda\vdash n} \chi_{\lambda}(w_\mu)\chi_{\lambda}(w_\nu)=\delta_{\mu\nu} |Z_{S_n}(w_\mu)|.
\end{equation}

If $\mu=\nu$, then we would get, using equation~\eqref{eq:Schur_scond}, that $2=|Z_{S_n}(w_\mu)|$.
Recall that $|Z_{S_n}(w_\mu)|=\prod_{i\geq 1} i^{m_i} m_i!$, where $m_i$ is the number of parts of $\mu$ equal to $i$.
Hence, we have $|Z_{S_n}(w_\mu)|=2$ if and only if $\mu=(2,1)$.
Since $n>10$, we must have $\mu\neq \nu$.
We may also conclude that the sign of the permutations $w_\mu$ and $w_\nu$ must be of different parity.

At this point we recall a lemma on product of conjugacy classes.
\begin{lemma}\cite[Corollary 4.14]{Navarro_book}
    \label{lemma:Navarro_product_of_conjugacy_classes}
    Let $G$ be a finite group and $C$ and $D$ be two conjugacy classes of $G$.
    Let $c\in C$ and $d\in D$.
    Then for an element $g\in G$, we have
    \begin{displaymath}
        |\{(x,y)\in C\times D : xy=g\}|=\dfrac{|C||D|}{|G|}\sum\limits_{\chi\in \irr(G)} \dfrac{\chi(c)\chi(d)\overline{\chi(g)}}{\chi(1)}.
    \end{displaymath}
\end{lemma}



Let us recall that the $w_\mu$ and $w_\nu$ have different parity.
Also, every non-linear irreducible character of $S_n$ vanishes on at least one of $w_\mu$ or $w_\nu$.
Using Lemma~\ref{lemma:Navarro_product_of_conjugacy_classes}, we have that the coefficient $a_{\mu\nu}^\gamma$ of $\tilde{C}_{\gamma}$ in $\tilde{C_\mu} \tilde{C_\nu}$ is given by
\begin{equation}
    a_{\mu\nu}^{\gamma}=
\begin{cases}
\dfrac{2|C_1||C_2|}{n!} & \text{if $w_\gamma$ is an odd permutation,}\\
0 & \text{otherwise.}
\end{cases}
\end{equation}

In particular, we have $a_{\mu\nu}^{(2)}>0$.
Surprisingly, this positivity of $a_{\mu\nu}^{(2)}$ puts a lot of constraints on the partitions $\mu$ and $\nu$.
Roughly speaking, the partitions $\mu$ and $\nu$ have to be almost equal.
Let $w_\mu$ and $w_\nu$ be the permutations in $S_n$ with cycle type $\mu$ and $\nu$ such that
\begin{equation}\label{eq:mu_nu_(12)}
w_\mu w_\nu =(1~2).
\end{equation}
Let $w_\mu=u_1\dotsc u_r$ and $w_\nu=v_1\dotsc v_m$ be the decomposition of $w_\mu$ and $w_\nu$ into disjoint cycles such that the length of $u_i$ is $\mu_i$ and the length of $v_j$ is $\nu_j$.
Let $u_\mu$ (resp. $v_\mu$) be the product of cycles of $w_\mu$ (resp. $w_\nu$) which contains $1$ or $2$.
Note that $u_\mu$ (resp. $v_\nu$) contains at most two cycles of $w_\mu$ (resp. $w_\nu$).
Let $S_\mu$ (resp. $S_\nu$) be the set of all indices the cycles of $u_\mu$ (resp. $v_\nu$) contain.
We claim that $S_\mu=S_\nu$.
If not, then one of the sets is not contained in the other.
We may start by assuming that $S_\nu$ is not contained in $S_\mu$.
Then there exists an element $t\in S_\nu$ such that $t\notin S_\mu$ and $v_\nu(t)\in S_\mu$.
Since $t\neq 1,2$, we must have $w_\mu w_\nu(t)=t$.
But
\begin{align*}
w_\mu w_\nu(t) &= w_\mu v_\nu(t)\\
&= u_\mu v_\nu(t) \in S_\mu.
\end{align*}
This implies that $w_\mu w_\nu(t)\neq t$ which is a contradiction.
Applying inverse to both sides of the equation~\eqref{eq:mu_nu_(12)} and switching the roles of $\mu$ and $\nu$, we have $S_\mu\subset S_\nu$.
This proves the claim.

Therefore, we have
\begin{align*}
    w_\mu w_\nu &= u_1\dots u_\mu \dots u_r v_1\dots v_\nu \dots v_m\\
    &= w_{\bar\mu} w_{\bar\nu} u_\mu v_\nu\\
    &= w_{\bar\mu} w_{\bar\nu} (1~2)\\
    &= e (1~2)\\
    &= (1~2),
\end{align*}
where $w_{\bar\mu}$ (resp. $w_{\bar\nu}$) is obtained from $w_\mu$ (resp. $w_\nu$) by removing the cycles of $u_\mu$ (resp. $v_\nu$).
Hence, we have $\bar\mu=\bar\nu$, where $\bar\mu$ (resp. $\bar\nu$) is the cycle type of $w_{\bar\mu}$ (resp. $w_{\bar\nu}$).
In particular, we have
\begin{equation}\label{eq:4_cycles}
    u_\mu v_\nu=(1~2).
\end{equation}

Now we shall conclude that the total number of cycles in $u_\mu$ and $v_\nu$ is exactly three.
Recall that $u_\mu$ (resp. $v_\nu$) contains at most two cycles.
Suppose that $u_\mu=u_i u_j$ (resp. $v_\nu=v_k v_l$) contains two cycles.
Then applying the sign character $\chi_{(1^n)}$ on both sides of the equation~\eqref{eq:4_cycles}, we get
\begin{align*}
    (-1)^{\mu_i-1}(-1)^{\mu_j-1}(-1)^{\nu_k-1}(-1)^{\nu_l-1} &= -1\\
    (-1)^{\mu_i+\mu_j+\nu_k+\nu_l-4} &= -1\\
    (-1)^{2\mu_i+2\mu_j-4} &= -1\\
    1 &= -1.
\end{align*}
This is absurd.
Similarly, one can show that the total number of cycles in $u_\mu$ and $v_\nu$ is not two.
Therefore, we conclude that the total number of cycles occurring in $u_\mu$ and $v_\nu$ is exactly three.
Hence, we proved the following lemma.
\begin{lemma}
    Let $\mu$ and $\nu$ be partitions such that $a_{\mu\nu}^{(2)}>0$ and $\nu>\mu$ in the dominance order.
    Then there are two parts $a,b$ of $\mu$ and a part $c$ of $\nu$ such that $a+b=c$ and $\bar\mu=\bar\nu$, where $\bar\mu$ (resp. $\bar\nu$) is obtained from $\mu$ (resp. $\nu$) by removing the parts $a,b$ (resp. $c$).
\end{lemma}
In order to prove our main theorem, it suffices to show the following lemma.
\begin{lemma}
    Let $\mu$ and $\nu$ be partitions of $n>10$.
    For positive integers $a$ and $b$, let $a$ and $b$ be parts of $\mu$ and $a+b$ be a part of $\nu$ such that $\bar\mu=\bar\nu$, where $\bar\mu$ (resp. $\bar\nu$) is obtained from $\mu$ (resp. $\nu$) by removing a part $a$ and a part $b$ (resp. a part $a+b$).
    Then there exists a non-linear irreducible character of $S_n$ which does not vanish on both $w_\mu$ and $w_\nu$ except when $\mu=(n-1,1)$ and $\nu=(n)$.
\end{lemma}
\begin{proof}
Let us assume that $\mu$ and $\nu$ are partitions of $n$ such that any non-linear irreducible character of $S_n$ vanishes on at least one of $w_\mu$ or $w_\nu$.
Now it suffices to show that $\mu=(n-1,1)$ and $\nu=(n)$.
For a partition $\lambda$ of $n$, let $m_i$ be the number of parts of $\lambda$ equal to $i$.
Let us recall a description of some irreducible characters of $S_n$ which we will use repeatedly in the proof. We refer the reader to~\cite[Section 8.1]{MR0002127} for the details of its proof.
\begin{itemize}
    \item $\chi_{(n-1,1)}(w_\mu)=m_1-1,$
    \item $\chi_{(n-2,2)}(w_\mu)=\dfrac{1}{2}m_1(m_1-3)+m_2,$
    \item $\chi_{(n-2,1,1)}(w_\mu)=\dfrac{1}{2}(m_1-1)(m_1-2)-m_2,$
    \item $\chi_{(n-3,3)}(w_\mu)=\dfrac{1}{6}m_1(m_1-1)(m_1-5)+m_2(m_1-1)+m_3,$
    \item $\chi_{(n-3,2,1)}(w_\mu)=\dfrac{1}{3}m_1(m_1-2)(m_1-4)-m_3,$
    \item $\chi_{(n-3,1,1,1)}(w_\mu)=\dfrac{1}{6}(m_1-1)(m_1-2)(m_1-3)-(m_1-1)m_2+m_3,$
    \item $\chi_{(n-4,2,1,1)}(w_\mu)=\dfrac{1}{8}m_1(m_1-2)(m_1-3)(m_1-5)-\dfrac{1}{2}m_2m_1(m_1-3)-\dfrac{1}{2}m_2(m_2-1)+m_4,$
    \item $\chi_{(n-4,1,1,1,1)}(w_\mu)=\dfrac{1}{24}(m_1-1)(m_1-2)(m_1-3)(m_1-4)-\dfrac{1}{2}(m_1-1)(m_1-2)m_2+(m_1-1)m_3+\dfrac{1}{2}m_2(m_2-1)-m_4,$
\end{itemize}

Without loss of generality, we may assume that $a\geq b$.
We have $m_t(\mu)=m_t(\nu)$ for all $t$ ($\neq a,b$), $m_a(\mu)=m_a(\nu)+1$, $m_b(\mu)=m_b(\nu)+1$, and $m_{a+b}(\mu)+1=m_{a+b}(\nu)$.
We shall prove the lemma by case-by-case analysis on the values of $a$ and $b$.
\subsection*{1} $b\geq 3$.

Let us apply the character $\chi_{(n-1,1)}$ on both permutations $w_\mu$ and $w_\nu$.
We obtain that $\chi_{(n-1,1)}(w_\mu)=m_1(\mu)-1=m_1(\nu)-1=\chi_{(n-1,1)}(w_\nu)$.
Since $\chi_{(n-1,1)}(w_\mu)=0$ or $\chi_{(n-1,1)}(w_\nu)=0$, we have $m_1(\mu)=m_1(\nu)=1$.
Similarly, applying the character $\chi_{(n-2,1,1)}$ on both permutations $w_\mu$ and $w_\nu$, we have
$m_2(\mu)=m_2(\nu)=0$.
In this case, the character $\chi_{(n-2,2)}$ does not vanish on both the permutations $w_\mu$ and $w_\nu$, a contradiction.

\subsection*{2} $b=2$ and $a>3$.

Note that $m_2(\mu)=m_2(\nu)+1>0$.
Like before, applying $\chi_{(n-1,1)}$ yields that $m_1(\nu)=m_1(\mu)=1$.
Since $\chi_{(n-2,1,1)}(w_\mu)=-m_2(\mu)\neq 0$, we must have $\chi_{(n-2,1,1)}(w_\nu)=-m_2(\nu) = 0$.
Therefore, $m_2(\mu)=1$ and $m_2(\nu)=0$.
Let us apply $\chi_{(n-3,3)}$, then we obtain $\chi_{(n-3,3)}(w_\mu)=m_3(\mu)=m_3(\nu)=\chi_{(n-3,3)}(w_\nu)$.
Our hypothesis on $\mu$ and $\nu$ implies that $m_3(\mu)=m_3(\nu)=0$.
In this case, the character $\chi_{(n-3,2,1)}$ does not vanish on both the permutations $w_\mu$ and $w_\nu$, a contradiction.

\subsection*{3} $b=2$ and $a=3$.

Note that $m_2(\mu)=m_1(\nu)+1>0$, $m_3(\mu)=m_3(\nu)+1>0$ and $m_5(\mu)+1=m_5(\nu)>0$.
Like in the previous case, applying $\chi_{(n-1,1)}$ and $\chi_{(n-2,1,1)}$ yields that $m_1(\nu)=m_1(\mu)=1$, $m_2(\mu)=1$, and $m_2(\nu)=0$.
Let us apply $\chi_{(n-3,3)}$, then we obtain $\chi_{(n-3,3)}(w_\mu)=m_3(\mu)>0$ and $\chi_{(n-3,3)}(w_\nu)=m_3(\nu)$.
Since one of them is zero, we have $m_3(\mu)=1$ and $m_3(\nu)=0$.
Now let us apply $\chi_{(n-4,1,1,1,1)}$ on both permutations $w_\mu$ and $w_\nu$, we have $\chi_{(n-4,1,1,1,1)}(w_\mu)=-m_4(\mu)=-m_4(\nu)=\chi_{(n-4,1,1,1,1)}(w_\nu)$.
Since one of them is zero and $m_4(\mu)=m_4(\nu)$, we must have $m_4(\mu)=m_4(\nu)=0$.
Recall that 
\begin{equation}\label{eq:recursion_for_hook}
\chi_{(n-k,1^k)}=\chi_{(n-k)}\times \chi_{(1^k)}\uparrow_{S_{n-k}\times S_k}^{S_n}-\chi_{(n-k+1,1^{k-1})}.
\end{equation}
Applying the character $\chi_{(n-5,1,1,1,1,1)}$ on both permutations $w_\mu$ and $w_\nu$, we have
\begin{align*}
    \chi_{(n-5,1,1,1,1,1)}(w_\mu)&=\chi_{(n-5)}\times \chi_{(1^5)}\uparrow_{S_{n-5}\times S_5}^{S_n}(w_\mu)-\chi_{(n-4,1^4)}(w_\mu)\\
    &=(-1)^{4}m_5(\mu)+(-1)-0\\
    &=m_5(\mu)-1.
\end{align*}
Similarly, we have $\chi_{(n-5,1,1,1,1,1)}(w_\nu)=m_5(\nu)>0$.
Hence, we must have $m_5(\mu)=1$ and therefore, $m_5(\nu)=2$.
In this case, the character $\chi_{(n-5,5)}$ does not vanish on both the permutations $w_\mu$ and $w_\nu$, a contradiction.
To see this, let us recall a recursive description for the irreducible characters corresponding to the two-row partitions;
\begin{equation*}\label{eq:recursion_for_two_row}
 \chi_{(n-k,k)}=\chi_{(n-k)}\times \chi_{(k)}\uparrow_{S_{n-k}\times S_k}^{S_n}-\chi_{(n)}-\chi_{(n-1,1)}-\chi_{(n-2,2)}-\chi_{(n-3,3)}-\dots-\chi_{(n-k+1,k-1)}.
\end{equation*}
Using the above equation, we have $\chi_{(n-4,4)}(w_\mu)=-1$ and $\chi_{(n-4,4)}(w_\nu)=0$.
Now let us apply the character $\chi_{(n-5,5)}$ on both permutations $w_\mu$ and $w_\nu$, we have $\chi_{(n-5,5)}(w_\mu)=1$ and $\chi_{(n-5,5)}(w_\nu)=2$, a contradiction.

\subsection*{4} $b=2$ and $a=2$.

Note that $m_2(\mu)=m_1(\nu)+2>0$ and $m_4(\mu)+1= m_4(\nu)>0$.
Like in the previous case, we have $m_1(\mu)=m_1(\nu)=1$.
Let us apply $\chi_{(n-2,1,1)}$, then we obtain $\chi_{(n-2,1,1)}(w_\mu)=m_2(\mu)$ and $\chi_{(n-2,1,1)}(w_\nu)=m_2(\nu)$.
Since one of them is zero, we have $m_2(\mu)=2$ and $m_2(\nu)=0$.
In this case, the character $\chi_{(n-2,2)}$ does not vanish on both the permutations $w_\mu$ and $w_\nu$, a contradiction.

\subsection*{5} $b=1$ and $a\geq 3$.

Note that $m_1(\mu)=m_1(\nu)+1>0$.
Let us apply $\chi_{(n-1,1)}$ on both permutations $w_\mu$ and $w_\nu$.
We obtain that $\chi_{(n-1,1)}(w_\mu)=m_1(\mu)-1$ and $\chi_{(n-1,1)}(w_\nu)=m_1(\nu)-1$.
Since one of them have to be zero, we have $m_1(\mu)=2$ or $1$.
\subsection*{5.1} $m_1(\mu)=2$ and $m_1(\nu)=1$.

Applying the character $\chi_{(n-2,2)}$ on both permutations $w_\mu$ and $w_\nu$, we have $\chi_{(n-2,2)}(w_\mu)=-1+m_2(\mu)$ and $\chi_{(n-2,2)}(w_\nu)=-1+m_2(\nu)$.
Since one of them is zero and $m_2(\mu)=m_2(\nu)$, we must have $m_2(\mu)=m_2(\nu)=1$.
In this case, the character $\chi_{(n-2,1,1)}$ does not vanish on both the permutations $w_\mu$ and $w_\nu$, a contradiction.

\subsection*{5.2} $m_1(\mu)=1$ and $m_1(\nu)=0$.

Let us apply $\chi_{(n-2,1,1)}$ on both permutations $w_\mu$ and $w_\nu$, we have $\chi_{(n-2,1,1)}(w_\mu)=-m_2(\mu)$ and $\chi_{(n-2,1,1)}(w_\nu)=1-m_2(\nu)$.
Since one of them is zero and $m_2(\mu)=m_2(\nu)$, we must have $m_2(\mu)=m_2(\nu)=1$ or $0$.
\subsection*{5.2.1} $m_2(\mu)=m_2(\nu)=1$ and $a>3$.

Let us apply $\chi_{(n-3,1,1,1)}$ on permutations $w_\mu$ and $w_\nu$, we have $\chi_{(n-3,1,1,1)}(w_\mu)=m_3(\mu)$ and $\chi_{(n-3,1,1,1)}(w_\nu)=m_3(\nu)$.
Since one of them is zero and $m_3(\mu)=m_3(\nu)$, we must have $m_3(\mu)=m_3(\nu)=0$.
For $2<i<a$, we shall prove that $m_i(\mu)=m_i(\nu)=0$ and $\chi_{(n-i,1^i)}(w_\mu) = 0 $ and $\chi_{(n-i,1^i)}(w_\nu) = 0$.
Clearly, the claim holds when $i=3$.
By induction, we may assume that the claim is true for $i<k<a$.
Now let us prove the claim for $i=k>3$.
Let us apply the character $\chi_{(n-k,1^k)}$ on both permutations $w_\mu$ and $w_\nu$, we have
\begin{align*}
    \chi_{(n-k,1^k)}(w_\mu)&=\chi_{(n-k)}\times \chi_{(1^k)}\uparrow_{S_{n-k}\times S_k}^{S_n} (w_\mu) - \chi_{(n-k+1,1^{k-1})}(w_\mu).\\
    &=(-1)^{k-1} m_k(\mu) -0 \text{ (by induction)},\\
    &=(-1)^{k-1}m_k(\mu).
\end{align*}
Similarly, we have $\chi_{(n-k,1^k)}(w_\nu)=(-1)^{k-1}m_k(\nu)$.
Since one of them is zero, we must have $m_k(\mu)=m_k(\nu)=0$.
Clearly, $\chi_{(n-k,1^k)}(w_\mu)=\chi_{(n-k,1^k)}(w_\nu)=0$.
This completes the claim and the induction.
Now let us apply the character $\chi_{(n-a,1^a)}$ on both permutations $w_\mu$ and $w_\nu$, we have
\begin{align*}
    \chi_{(n-a,1^a)}(w_\mu)&=\chi_{(n-a)}\times \chi_{(1^a)}\uparrow_{S_{n-a}\times S_a}^{S_n} (w_\mu) - \chi_{(n-a+1,1^{a-1})}(w_\mu)\\
    &=(-1)^{a-1} m_a(\mu) -0 \text{ (by the claim)},\\
    &=(-1)^{a-1}m_a(\mu).
\end{align*}
Similarly, we have $\chi_{(n-a,1^a)}(w_\nu)=(-1)^{a-1}m_a(\nu)$.
Since one of them is zero and $m_a(\mu)=m_a(\nu)+1>0$, we must have $m_a(\nu)=0$.
This implies that $m_a(\mu)=1$.
Finally, we may use the character $\chi_{(n-a-1,1^{a+1})}$ on both permutations $w_\mu$ and $w_\nu$, we have
\begin{align*}
    \chi_{(n-a-1,1^{a+1})}(w_\mu)&=\chi_{(n-a-1)}\times \chi_{(1^{a+1})}\uparrow_{S_{n-a-1}\times S_{a+1}}^{S_n} (w_\mu) - \chi_{(n-a,1^a)}(w_\mu)\\
    &=(-1)^{a} m_{a+1}(\mu)+(-1)^{a-1} -(-1)^{a-1}\\
    &=(-1)^{a}m_{a+1}(\mu).
\end{align*}
Similarly, we have 
\begin{align*}
    \chi_{(n-a-1,1^{a+1})}(w_\nu)&=\chi_{(n-a-1)}\times \chi_{(1^{a+1})}\uparrow_{S_{n-a-1}\times S_{a+1}}^{S_n} (w_\nu) - \chi_{(n-a,1^a)}(w_\nu)\\
    &=(-1)^{a} m_{a+1}(\nu) -0\\
    &=(-1)^{a}m_{a+1}(\nu).
\end{align*}
Since one of them is zero and $m_{a+1}(\mu)+1=m_{a+1}(\nu)>0$, we must have $m_{a+1}(\mu)=0$ and $m_{a+1}(\nu)=1$.
We can go on to prove that all the parts of $\mu$ and $\nu$ which are greater than $a+1$ are equal and occurs with the multiplicity $1$.
But this is not necessary as we already have $\chi_{(n-4,2,2)}(w_\mu)\neq 0$ and $\chi_{(n-4,2,2)}(w_\nu)\neq 0$ by Murnaghan-Nakayama rule.
To see this, notice that all the parts of $\mu$ (resp. $\nu$) other than $1,2$ (resp. $2$) are greater than or equal to $a$ (resp. $a+1$) which is greater than or equal to $4$ (resp. $5$).
Also, $m_a(\mu)=m_{a+1}(\nu)=1$ and $m_{a+1}(\mu)=m_a(\nu)=0$.
Let us fix the part $a$ of $\mu$.
This splits the partition $\mu$ to a partition $(a,2,1)$ of $a+2+3$ and a partition $\mu^1$ of $n-a-3$, where $\mu^1$ is obtained from $\mu$ by removing the parts $a,2,1$.
Recall that all the parts of $\mu^1$ are greater than or equal to $5$.
Now we may apply Murnaghan-Nakayama rule.
If $\mu^1$ has at least a part, then there exists a unique cell (in the first row) of $(n-4,2,2)$ with hook length $\mu^1_1$.
Hence, we have
\begin{align*}
    \chi_{(n-4,2,2)}(w_{(\mu^1 \cup (a,2,1))})&=\chi_{(n-4-\mu^1_1,2,2)}(w_{\mu^2\cup (a,2,1)}),
\end{align*}
where $\mu^2$ is the partition obtained from $\mu^1$ by removing the part $\mu^1_1$.
Using induction, we may conclude that $\chi_{(n-4-\mu^2_1,2,2)}(w_{\mu^2\cup (a,2,1)})=\chi_{(a-1,2,2)}(w_{(a,2,1)})$.
Since there is a unique cell (in the first row) of $(a-1,2,2)$ with hook length $a$, we have $\chi_{(a-1,2,2)}(w_{(a,2,1)})=\chi_{(1,1,1)}(w_{(a,2,1)})=-1\neq 0$.
Similarly, we can show that $\chi_{(n-4,2,2)}(w_{\nu})=-1$.
If $\mu^1$ has no part, then also we have $\chi_{(n-4,2,2)}(w_{(a,2,1)})=-1$ and $\chi_{(n-4,2,2)}(w_{(a+1,2)})=-1$.
This is a contradiction to the assumption that every non-linear irreducible character of $S_n$ vanishes on at least one of $w_\mu$ or $w_\nu$.

\subsection*{5.2.2} $m_2(\mu)=m_2(\nu)=1$ and $a=3$.

Note that $m_3(\mu)=m_3(\nu)+1>0$ and $m_4(\mu)+1=m_4(\nu)>0$.
Let us apply $\chi_{(n-3,1,1,1)}$ on both permutations $w_\mu$ and $w_\nu$, we have $\chi_{(n-3,1,1,1)}(w_\mu)=m_3(\mu)$ and $\chi_{(n-3,1,1,1)}(w_\nu)=m_3(\nu)$.
Since one of them is zero and $m_3(\mu)>0$, we must have $m_3(\nu)=0$ and hence $m_3(\mu)=1$.
Let us apply $\chi_{(n-4,1,1,1,1)}$ on both permutations $w_\mu$ and $w_\nu$, we have $\chi_{(n-4,1,1,1,1)}(w_\mu)=-m_4(\mu)$ and $\chi_{(n-4,1,1,1,1)}(w_\nu)=-m_4(\nu)$.
Since one of them is zero and $m_4(\nu)>0$, we must have $m_4(\mu)=0$ and hence $m_4(\nu)=1$.
Now it is easy to see that the character $\chi_{(n-4,2,2)}$ does not vanish on both the permutations $w_\mu$ and $w_\nu$ by Murnaghan-Nakayama rule as we did in the previous case, a contradiction.

\subsection*{5.2.3} $m_2(\mu)=m_2(\nu)=0$ and $a>3$.

Let us apply the character $\chi_{(n-3,3)}$ on both permutations $w_\mu$ and $w_\nu$, we have $\chi_{(n-3,3)}(w_\mu)=m_3(\mu)=m_3(\nu)=\chi_{(n-3,3)}(w_\nu)$.
Since one of them is zero and $m_3(\mu)=m_3(\nu)$, we must have $m_3(\mu)=m_3(\nu)=0$.
Let $t=\min(a,\lfloor n/2 \rfloor)$.
For $3\leq k <t$, we shall prove that $\chi_{(n-k,k)}(w_\mu)=0=\chi_{(n-k,k)}(w_\nu)$ and $m_k(\mu)=m_k(\nu)=0$.
The base case $k=3$ is already done.
By induction, we may assume that the claim is true for $3\leq i<k<t$.
Now let us prove the claim for $k$.
Let us apply the character $\chi_{(n-k,k)}$ on both permutations $w_\mu$ and $w_\nu$, we have
\begin{align*}
    \chi_{(n-k,k)}(w_\mu)&=\chi_{(n-k)}\times \chi_{(k)}\uparrow_{S_{n-k}\times S_k}^{S_n} (w_\mu) - \chi_{(n)}(w_\mu) - \chi_{(n-1,1)}(w_\mu)\\
    &\quad -\chi_{(n-2,2)}(w_\mu)-\chi_{(n-3,3)}(w_\mu)-\dots-\chi_{(n-k,k)}(w_\mu)\\
    &=(-1)^{k-1} m_k(\mu) -1 -0 -(-1) -0 -\dotsc-0\\
    &=(-1)^{k-1}m_k(\mu).
\end{align*}
Similarly, we have $\chi_{(n-k,k)}(w_\nu)=(-1)^{k-1}m_k(\nu)$.
Since one of them is zero and $m_k(\mu)=m_k(\nu)$, we must have $m_k(\mu)=m_k(\nu)=0$.
Clearly, $\chi_{(n-k,k)}(w_\mu)=\chi_{(n-k,k)}(w_\nu)=0$.
This completes the induction and proves the claim.
Suppose that $a>t$, then we must have $m_k(\mu)=m_k(\nu)=0$ for all $3\leq k\leq t$.
Indeed, $m_t(\mu)=m_t(\nu)=0$ follows from the proof of the above claim.
In this case, there can be only one more part in $\mu$ and that is equal to $n-1$.
Hence, $\mu=(n-1,1)$ and $\nu=(n)$.

Otherwise, we have $t=a$.
Let us apply the character $\chi_{(n-a,a)}$ on both permutations $w_\mu$ and $w_\nu$, we have $\chi_{(n-a,a)}(w_\mu)=m_a(\mu)$ and $\chi_{(n-a,a)}(w_\nu)=m_a(\nu)$.
Since one of them is zero and $m_a(\mu)=m_a(\nu)+1>0$, we must have $m_a(\nu)=0$ and $m_a(\mu)=1$.
This implies that $m_{a+1}(\nu)>0$ and $m_k(\nu)=0$ for all $k< a+1$.
This is impossible as all the parts of $\nu$ other than $a+1$ are less than $a+1$.

\subsection*{5.2.4} $m_2(\mu)=m_2(\nu)=0$ and $a=3$.

Note that $m_3(\mu)=m_3(\nu)+1>0$ and $m_4(\mu)+1=m_4(\nu)>0$.
Let us apply the character $\chi_{(n-3,3)}$ on both permutations $w_\mu$ and $w_\nu$, we have $\chi_{(n-3,3)}(w_\mu)=m_3(\mu)$ and $\chi_{(n-3,3)}(w_\nu)=m_3(\nu)$.
Since one of them is zero and $m_3(\mu)>0$, we must have $m_3(\nu)=0$ and hence $m_3(\mu)=1$.
Let us apply the character $\chi_{(n-4,4)}$ on both permutations $w_\mu$ and $w_\nu$, we have $\chi_{(n-4,4)}(w_\mu)=m_4(\mu)+1-1-0-(-1)-1=m_4(\mu)$ and $\chi_{(n-4,4)}(w_\nu)=m_4(\nu)-1-(-1)-0-0=m_4(\nu)$.
Since one of them is zero and $m_4(\nu)>0$, we must have $m_4(\mu)=0$ and hence $m_4(\nu)=1$.
In this case, the character $\chi_{(n-4,2,1,1)}$ does not vanish on both the permutations $w_\mu$ and $w_\nu$, a contradiction.

\subsection*{6} $b=1$ and $a=2$.

Note that $m_1(\mu)=m_1(\nu)+1>0$, $m_2(\mu)=m_2(\nu)+1>0$ and $m_3(\mu)+1=m_3(\nu)>0$.
Let us apply the character $\chi_{(n-1,1)}$ on both permutations $w_\mu$ and $w_\nu$, we have $\chi_{(n-1,1)}(w_\mu)=m_1(\mu)-1$ and $\chi_{(n-1,1)}(w_\nu)=m_1(\nu)-1$.
Since one of them is zero, we must have $m_1(\mu)=2$ or $1$.

\subsection*{6.1} $m_1(\mu)=2$ and $m_1(\nu)=1$.

Let us apply the character $\chi_{(n-2,1,1)}$ on both permutations $w_\mu$ and $w_\nu$, we have $\chi_{(n-2,1,1)}(w_\mu)=-m_2(\mu)$ and $\chi_{(n-2,1,1)}(w_\nu)=-m_2(\nu)$.
Since one of them is zero and $m_2(\mu)>0$, we must have $m_2(\nu)=0$ and hence $m_2(\mu)=1$.
Now we may apply the character $\chi_{(n-3,1,1,1)}$ on both permutations $w_\mu$ and $w_\nu$, we have $\chi_{(n-3,1,1,1)}(w_\mu)=m_3(\mu)-1$ and $\chi_{(n-3,1,1,1)}(w_\nu)=m_3(\nu)$.
Since one of them is zero and $m_3(\nu)>0$, we must have $m_3(\mu)=1$ and hence $m_3(\nu)=2$.
In this case, the character $\chi_{(n-3,3)}$ does not vanish on both the permutations $w_\mu$ and $w_\nu$, a contradiction.

\subsection*{6.2} $m_1(\mu)=1$ and $m_1(\nu)=0$.

Let us apply the character $\chi_{(n-2,1,1)}$ on both permutations $w_\mu$ and $w_\nu$, we have $\chi_{(n-2,1,1)}(w_\mu)=-m_2(\mu)$ and $\chi_{(n-2,1,1)}(w_\nu)=1-m_2(\nu)$.
Since one of them is zero and $m_2(\mu)>0$, we must have $m_2(\nu)=1$ and hence $m_2(\mu)=2$.
In this case, the character $\chi_{(n-2,2)}$ does not vanish on both the permutations $w_\mu$ and $w_\nu$, a contradiction.

\subsection*{7} $b=1$ and $a=1$.

Note that $m_1(\mu)=m_1(\nu)+2>0$ and $m_2(\mu)+1=m_2(\nu)>0$.
Applying the character $\chi_{(n-1,1)}$ on both permutations $w_\mu$ and $w_\nu$ yields that $m_1(\mu)=3$ and $m_1(\nu)=1$.
Let us apply the character $\chi_{(n-2,1,1)}$ on both permutations $w_\mu$ and $w_\nu$, we have $\chi_{(n-2,1,1)}(w_\mu)=1-m_2(\mu)$ and $\chi_{(n-2,1,1)}(w_\nu)=-m_2(\nu)$.
Since one of them is zero and $m_2(\nu)>0$, we must have $m_2(\mu)=1$ and hence $m_2(\nu)=2$.
In this case, the character $\chi_{(n-2,2)}$ does not vanish on both the permutations $w_\mu$ and $w_\nu$, a contradiction.

This completes the proof.
\end{proof}
\section*{Acknowledgements}
I thank Rijubrat Kundu for introducing me to the problem and for his fruitful discussions.
I would like to thank my advisor, Amritanshu Prasad for his constant support, encouragement and for his valuable comments on the manuscript.
I would also like to thank G. Arunkumar, R. Balasubramanian, Anup Dixit, Jyotirmoy Ganguly, Subhajit Ghosh, S. Sundar, Sankaran Viswanath for their encouragement and insights.

\bibliographystyle{abbrv}
\bibliography{refs}

\begin{thebibliography}{1}

\bibitem{Hung}
N.~N. Hung, A.~Moret\'o, and L.~Morotti.
\newblock Common zeros of irreducible characters.
\newblock {\em J. Aust. Math. Soc.}, 117(2):105--129, 2024.

\bibitem{Isaacs_I_Martin}
I.~M. Isaacs.
\newblock {\em Character theory of finite groups}.
\newblock AMS Chelsea Publishing, Providence, RI, 2006.
\newblock Corrected reprint of the 1976 original [Academic Press, New York;
  MR0460423].

\bibitem{James_Kerber}
G.~James and A.~Kerber.
\newblock {\em The representation theory of the symmetric group}, volume~16 of
  {\em Encyclopedia of Mathematics and its Applications}.
\newblock Addison-Wesley Publishing Co., Reading, MA, 1981.
\newblock With a foreword by P. M. Cohn, With an introduction by Gilbert de B.
  Robinson.

\bibitem{Frieder_Ladisch}
F.~Ladisch.
\newblock Groups with anticentral elements.
\newblock {\em Comm. Algebra}, 36(8):2883--2894, 2008.

\bibitem{Mark_Lewis_L}
M.~L. Lewis.
\newblock Camina groups, {C}amina pairs, and generalizations.
\newblock In {\em Group theory and computation}, Indian Stat. Inst. Ser., pages
  141--173. Springer, Singapore, 2018.

\bibitem{Mark_Lucia_Emanuele_Lucia_Hung}
M.~L. Lewis, L.~Morotti, E.~Pacifici, L.~Sanus, and H.~P. Tong-Viet.
\newblock On common zeros of characters of finite groups.
\newblock {\em Algebras and Representation Theory}, 2025.

\bibitem{MR0002127}
D.~E. Littlewood.
\newblock {\em The {T}heory of {G}roup {C}haracters and {M}atrix
  {R}epresentations of {G}roups}.
\newblock Oxford University Press, New York, 1940.

\bibitem{Navarro_book}
G.~Navarro.
\newblock {\em Character theory and the {M}c{K}ay conjecture}, volume 175 of
  {\em Cambridge Studies in Advanced Mathematics}.
\newblock Cambridge University Press, Cambridge, 2018.

\end{thebibliography}

\end{document}